\documentclass[10pt]{article}
   \usepackage{graphicx}
\usepackage{latexsym}
\usepackage{float}
\usepackage{amsmath}
\usepackage{cite}
\usepackage{amssymb}
\usepackage{amsthm}
\usepackage{url}
\usepackage{tabulary}
\usepackage{booktabs}
 \usepackage{helvet}

\theoremstyle{plain}
\newtheorem{theorem}{Theorem}[section]

\theoremstyle{definition}

\def \Z {{\mathbb Z}}

\title{Primes dividing values of a given polynomial }

\begin{document}
	 
	\maketitle
	\vspace*{-0.9 cm}

\begin{abstract}Let  $P(x) \in \Z[x]$ be a polynomial. We give an easy and  new proof of the fact that the set of primes $p$ such that $p \mid P(n)$, for some $n \in \Z$, is infinite.  
We also get analog  of this result for some special domains.
\end{abstract}

\maketitle

\section{Main results} 
 Let  $P(x) \in \Z[x]$ be a polynomial.  Consider the set of      primes $p$ such that $p \mid f(n)$, for some $n \in \Z$. It  is known that this set is  infinite. This was proved for the first time by Schur and is called as Schur's theorem (see Schur   \cite{MR0462891}
 ).  In this article, we give a  new proof of this fact.      For a given prime $p$ and given number $d,$ we denote the highest power of $p$ dividing $d$ by $w_p(d)$. For instance, $w_2(12)=2^2.$ Now, we state our main result.
  
  \begin{theorem}\label{math}
  
  Let  $P(x) \in \Z[x]$ be a polynomial.  Then the set of      primes $p,$ such that there exists $n \in \Z$ such that $p \mid P(n)$   is infinite.
   
  \end{theorem}
  
  \begin{proof} We prove by contradiction. Assume there are only finitely many 
 primes. We can label them as $p_1,p_2, \ldots, p_r$. There exist integers $k_i$ and $e_i$ such that   $w_{p_i}(f(k_i))=p_i^{e_i}$  $ \forall\ 1 \leq i \leq r$.  
   If $n_0$ is a solution of congruence equations $x\equiv {k_i}  \pmod {p_i^{e_i+1}}$, 
 then $f(n_0)$ is divisible by  $p_i^{e_i}$ but is not divisible by $p_i^{e_i+1}$ for any $0 \leq  i  \leq r $. This holds since   $m \equiv n \pmod  {\prod p^{
ei+1}_i}
$ implies $P(m) \equiv  P(n) \pmod  {\prod p^{
ei+1}_i},$ so that if $w_{p_i}(P(m)) = p_i^{e_i},$    
then $w_{p_i}(P(n)) $ is also $ p_i^{e_i}$.  Since $n_0$ is a solution  of the above congruence,   $n = n_0 + k
\prod
p_i^{e_i+1} $ is also a solution for every integer $ k .$ 
  The equality $P(n)= \prod_{i=1}^r p_i^{e_i}$ or $P(n)= - \prod_{i=1}^r p_i^{e_i}$  can hold only for finitely many solutions $n$  of $x\equiv {k_i}  \pmod {p_i^{e_i+1}}$ 
 Hence,  there exist a solution  $n_k$ of $x\equiv {k_i}  \pmod {p_i^{e_i+1}}$  such that $\prod_{i=0}^r P_i^{e_i}$   is a proper divisor of $P(n_k)$. Since $p_i^{e_i+1}$ cannot divide $P(n_k)$, so there exists a prime other than $p_1,p_2, \ldots p_r$ dividing $P(n_k)$, which is a contradiction. Hence the result follows.\end{proof}

  By our approach, it is evident that the result holds for polynomials with coefficients in a Dedekind domain or sometimes a domain.

\noindent \textbf{Acknowledgement:}
We thank  the anonymous referee for the comments which improved the presentation of the paper.

\bigskip

\bigskip

\noindent \textsc{Devendra Prasad \\ Department of Mathematics\\ IISER-Tirupati, Tirupati \\ Andhra Pradesh  517507, India.}\\
E-mail: devendraprasad@iisertirupati.ac.in

\end{document}